\documentclass{aptpub}
\numberwithin{equation}{section}
\authornames{Z.-C. Hu and B. Jiang} % insert the authors here for use in running head
\shorttitle{Joint ruin probabilities with constant interest rate} % insert short title here for use in running head

\begin{document}

\title{On joint ruin probabilities of a two-dimensional risk model with constant interest rate} % insert title - use \\ if it requires more than one line.

\authorone[Nanjing University]{ZECHUN HU} % Affiliation is just the name of your university or institution
\authortwo[Monash University]{BIN JIANG}
\addressone{Department of Mathematics, Nanjing University, Nanjing 210093, China.\\
Email address: huzc@nju.edu.cn} % Your postal address goes here.

\addresstwo{Department of Econometrics and Business Statistics, Monash University, Clayton, Victoria 3800, Australia. Email address: bin.jiang@monash.edu}

\begin{abstract}
% text of abstract goes here!
In this note we consider the two-dimensional risk model
introduced in Avram et al. \cite{APP08} with constant interest rate.
We derive the  integral-differential equations of the Laplace transforms, and asymptotic
expressions for the finite time ruin probabilities  with respect to the
joint ruin times $T_{\rm max}(u_1,u_2)$  and $T_{\rm min}(u_1,u_2)$ respectively.
\end{abstract}

\smallskip

 % insert keywords separated by a semicolon

\smallskip

 % insert the primary Maths Subject Classification number in the first bracket
         % and the secondary ams number(s) in the second bracket
         % e.g. \ams{60E20}{49G03;49F10}

\section{Introduction and Preliminaries} % Initial capital letter, then lower case. No full stop.
Ruin theory for the univariate risk model has been studied
extensively, see Asmussen \cite{As00}, Rolski et al. \cite{Ro99} and
many recent papers. Contrarily, there are only a few research on
multivariate risk models. Chan et al. \cite{CYZ03} studied the
following two-dimensional risk model
\begin{eqnarray*}
\left(\begin{array}{l}U_1(t)\\U_2(t)\end{array}\right)
=\left(\begin{array}{l}u_1\\u_2\end{array}\right)+\left(\begin{array}{l}c_1\\c_2\end{array}\right)t-
\sum_{j=1}^{N(t)}\left(\begin{array}{l}X_{1j}\\X_{2j}\end{array}\right),
\end{eqnarray*}
where for fixed $i=1$ or 2, $\{X_{ij},j=1,2,\ldots\}$ are i.i.d.
claim size random variables, $\{X_{1j},j=1,2,\ldots\}$ and
$\{X_{2j},j=1,2,\ldots\}$ are independent, and both of them are also
independent of the Poisson process $N(t)$.

Cai and Li \cite{CL05} studied the multivariate risk model
\begin{eqnarray}\label{CL}
\left(\begin{array}{c}U_1(t)\\\vdots\\U_s(t)\end{array}\right)
=\left(\begin{array}{c}u_1+p_1t-\sum_{n=1}^{N(t)}X_{1,n}\\\vdots\\
u_s+p_st-\sum_{n=1}^{N(t)}X_{s,n}\end{array}\right),
\end{eqnarray}
where $\{(X_{1,n},\ldots,X_{s,n}),n\geq 1\}$ is a sequence of i.i.d.
non-negative random vectors, and independent of  the Poisson process
$N(t)$. The model (\ref{CL}) was further studied by Cai and Li in
\cite{CL07}.

Yuen et al. \cite{YGW06} discussed the bivariate compound
Poisson model
\begin{eqnarray*}
\left(\begin{array}{l}U_1(t)\\U_2(t)\end{array}\right)
=\left(\begin{array}{l}u_1\\u_2\end{array}\right)+\left(\begin{array}{l}c_1\\c_2\end{array}\right)t-
\left(\begin{array}{l}\sum\limits_{i=1}^{M_1(t)+M(t)}X_i\\\sum\limits_{i=1}^{M_2(t)+M(t)}Y_i\end{array}\right),
\end{eqnarray*}
where $M_1(t), M_2(t)$ and $M(t)$ are three independent Poisson
processes, $X_i(Y_i)$ are i.i.d. claim size random variables,
$\{X_i,i\geq 1\}$ and $\{Y_i,i\geq 1\}$ are independent and they are
independent of the three Poisson processes.

Li et al. \cite{LLT07} discussed the bidimemsional
perturbed risk model
\begin{eqnarray*}
\left(\begin{array}{l}U_1(t)\\U_2(t)\end{array}\right)
=\left(\begin{array}{l}u_1\\u_2\end{array}\right)+\left(\begin{array}{l}c_1\\c_2\end{array}\right)t-
\sum_{j=1}^{N(t)}\left(\begin{array}{l}X_{1j}\\X_{2j}\end{array}\right)+
\left(\begin{array}{l}\sigma_1B_1(t)\\\sigma_2B_2(t)\end{array}\right),
\end{eqnarray*}
where $N(t)$ is a Poisson process, $\{(X_{1j},X_{2j}),j\geq 1\}$ is
a sequence of i.i.d. random vectors, $(B_1(t),B_2(t))$ is a standard
bidimensional Brownian motion, and the three processes are mutually
independent.

Avram et al. \cite{APP08} studied the two-dimensional risk
model below
\begin{eqnarray}\label{model}
\left(\begin{array}{l}U_1(t)\\U_2(t)\end{array}\right)
=\left(\begin{array}{l}u_1\\u_2\end{array}\right)+\left(\begin{array}{l}c_1\\c_2\end{array}\right)t-
\left(\begin{array}{l}\delta_1\\\delta_2\end{array}\right)S(t),
\end{eqnarray}
where  $S(t)$ is a L\'{e}vy process with only upward jumps that
represents the cumulative amount of claims up to time $t$, and the
paper focuses on the classic Cram\'{e}r-Lundberg model, i.e. $S(t)$
is a compound Poisson process.

In this note, we discuss the above two-dimensional risk
model (\ref{model}) with constant interest rate. About univariate ruin models with investment income, there have been a lot
of research. Please refer to the recent survey paper Paulsen
\cite{Pa08} and the references therein.

Now we introduce our model. Let $r$ be a nonnegative constant, which
represents the interest rate. Then our model can be expressed as
follows:
\begin{eqnarray}\label{model-1}
U_i(t)=e^{r t} u_i+c_i \int_{0}^{t} e^{r (t-v)}dv -\delta_i
\int_{0}^{t} e^{r (t-v)}d S_v, \  i=1, 2,
\end{eqnarray}
where $u_i$ are the initial reserves, $c_i$ are the premium rates,
and $0<\delta_1,\delta_2<1$ with $\delta_1+\delta_2=1$.
 $S_t$ is
taken to be a compound Poisson process, i.e.
$S_t=\sum\limits_{k=1}^{N(t)}\sigma_k,\  t \geq 0, $ where $N(t)$ is
a Poisson process with  intensity $\lambda>0$ and $\{\sigma_k,
k\geq1\}$ is a sequence of i.i.d. random variables independent of
$N(t)$. Denote by $F$ the distribution function, by $f$ the
probability density function of $\sigma_k$, respectively. Let
$\theta_k$ be the arrival time of the $k$-th claim. Then we can
rewrite (\ref{model-1}) as
\begin{eqnarray}\label{model-2}
U_i(t)=e^{r t} u_i+\frac{c_i}{r} (e^{r t}-1)-\delta_i
\sum\limits_{k=1}^{N(t)} e^{r (t-\theta_k)} \sigma_k,\  i=1, 2.
\end{eqnarray}
For $k=1,2,\ldots,$ denote by $T_k$ the inter-time between the
$(k-1)$-th claim and the $k$-th claim. Then $\{T_k,k\geq 1\}$ is a
sequence of i.i.d. random variables with exponential distribution
with the parameter $\lambda$, and $ \theta_k=\sum_{i=1}^kT_i. $

Define  two joint ruin times by
\begin{eqnarray*}
&&T_{\rm min}(u_1,u_2):=\inf\{t \geq 0|\min\{U_1(t),U_2(t)\} < 0\},\\
&&T_{\rm max}(u_1,u_2):=\inf\{t \geq 0|\max\{U_1(t),U_2(t)\} < 0\},
\end{eqnarray*}
and the corresponding  ruin probabilities
\begin{eqnarray*}
&&\psi_{\rm min}(u_1,u_2):=P\{T_{\rm min}(u_1,u_2) < \infty \},\\
&&\psi_{\rm max}(u_1,u_2):=P\{T_{\rm max}(u_1,u_2) < \infty \}.
\end{eqnarray*}

As in \cite{APP08}, we assume that $c_1/\delta_1 > c_2/\delta_2.$
Then if  $u_1/\delta_1 > u_2/\delta_2$, the above two joint ruin
probabilities degenerate into one-dimensional ruin probabilities as
follows:
\begin{eqnarray*}
&&\psi_{\rm min}(u_1,u_2)=\psi_{2}(u_2):=P\{\exists\, t<\infty\ s.t.\ U_2(t)<0\},\\
&&\psi_{\rm max}(u_1,u_2)=\psi_{1}(u_1):=P\{\exists\, t<\infty\ s.t.\
U_1(t)<0\}.
\end{eqnarray*}
Refer to \cite{APP08} for the deduction. Throughout the rest of this note, we
assume that $c_1/\delta_1 > c_2/\delta_2$ and  $u_1/\delta_1\leq
u_2/\delta_2$.

\begin{rem}
For each i, we know that
\begin{eqnarray*}
U_i(t)=e^{r t} u_i+c_i \int_{0}^{t} e^{r (t-v)}dv -\delta_i
\int_{0}^{t} e^{r (t-v)}d S_v=e^{r t} u_i+\int_{0}^{t}e^{r
(t-v)}d(c_i v-\delta_i S_v).
\end{eqnarray*}
Denote $\vec{U}(t)=(U_1(t),U_2(t))$, $\vec{u}=(u_1,u_2)$ and
$\vec{Z}_v=(c_1 v-\delta_1 S_v, c_2 v-\delta_2 S_v)$. Then we have
\begin{eqnarray}\label{rem2.1-a}
\vec{U}(t)= e^{r t} \vec{u}+ \int_{0}^{t}e^{r (t-v)}d \vec{Z}_v=e^{r
t}\left(\vec{u}+ \int_{0}^{t}e^{-rv}d \vec{Z}_v\right).
\end{eqnarray}
Differentiating both  sides of (\ref{rem2.1-a}) relative to $t$, we
obtain
\begin{eqnarray}\label{rem2.1-b}
d\vec{U}(t)&=&re^{rt}\left(\vec{u}+\int_0^t e^{-rv}d\vec{Z}_v\right)dt+e^{rt}e^{-rt}d\vec{Z}_t\nonumber\\
&=&r\vec{U}(t)dt+d\vec{Z}_t.
\end{eqnarray}
Integrating  both  sides of (\ref{rem2.1-b}) relative to $t$, we get
\begin{eqnarray}\label{rem2.1-c}
\vec{U}(t)=\vec{U}(0)+r\int_0^t\vec{U}(s)ds+\int_0^td\vec{Z}_s.
\end{eqnarray}
By (\ref{rem2.1-c}) and  the fact that
$(t,\vec{Z}_t)=(t,c_1t-\delta_1 S(t),c_2t-\delta_2 S(t))$ are a
three-dimensional L\'{e}vy process, following Protter \cite[Theorem 32]{Pr92}, we know that $\vec{U}(t)$ is a two-dimensional
homogeneous strong Markov process.
\end{rem}

The rest of this note is organized as follows. In Section 2, we show
the integral-differential equations of the Laplace transforms of the
joint ruin times $T_{\rm min}(u_1,u_2)$ and $T_{\rm max}(u_1,u_2)$ respectively. In Section 3, we provide two
asymptotic expressions for the finite time ruin probabilities with respect to
the joint ruin time $T_{\rm max}(u_1,u_2)$ and $T_{\rm min}(u_1,u_2)$ respectively.

% Write the text of your paper using normal LaTeX commands.
% For instance, you can use the `\cite' command~\cite{ref1}.
% When giving citations a numbering system is preferred~\cite{ref2},
% but an author--date system is also acceptable~\cite{ref3}.
\section{Integral-differential equation}
In this section, we establish the integral-differential equations of the Laplace transforms of the joint ruin times $T_{\rm min}(u_1,u_2)$ and $T_{\rm max}(u_1,u_2)$ respectively.

\subsection{The result about $T_{\rm min}(u_1,u_2)$}
In this subsection, we consider the joint ruin time $T_{\rm min}(u_1,u_2)$.
For convenience, we denote $T_{\rm min}(u_1,u_2)$ by $\tau(u_1,u_2)$.
Its Laplace transform is defined by
\begin{eqnarray}
\Psi_{\rm min}(u_1,u_2,s):= E\left[e^{-s \tau(u_1,u_2)}\right], \quad\mbox{for}\ s>0.
\end{eqnarray}
Then
\begin{eqnarray}\label{bound}
0\leq \Psi_{\rm min}(u_1,u_2,s)\leq 1.
\end{eqnarray}
Now we have the following result.

\begin{thm}\label{thm3.1}
 For $\frac{u_1}{\delta_1} \leq \frac{u_2}{\delta_2}$ and $s>0$, the function $\Psi_{\rm min}(\cdot,\cdot,s)$ satisfies the following integral-differential equation
\begin{eqnarray}\label{thm3.1-0}
&&\left(u_1+\frac{c_1}{r}\right) \frac{\partial \Psi_{\rm min}}{\partial
u_1}+\left(u_2+\frac{c_2}{r}\right) \frac{\partial
\Psi_{\rm min}}{\partial u_2}-\frac{\lambda+s}{r}
\Psi_{\rm min}\nonumber\\
&&\hskip 2.8cm+\frac{\lambda}{r} \int_{0}^{\infty}
\Psi_{\rm min}(u_1-\delta_1 z,u_2-\delta_2 z,s)f(z)d z=0
\end{eqnarray}
with the boundary condition
\begin{eqnarray}\label{thm3.1-k}
\Psi_{\rm min}(u_1,\frac{\delta_2}{\delta_1}u_1,s)= E\left[e^{-s \tau_2(\frac{\delta_2}{\delta_1}u_1)}\right],
\end{eqnarray}
where $f(z)$ is the probability density function of $\sigma_k$ and $\tau_2$ is the ruin time of risk process $U_2(t)$. Furthermore, $\Psi_{\rm min}$ is the unique solution of (\ref{thm3.1-0})-(\ref{thm3.1-k}).
\end{thm}
\begin{proof} {\bf Existence:} For any $h>0$, by considering the occurrence
time $T_1$ of the first claim, we have
\begin{eqnarray}\label{thm3.1-a}
E[e^{-s \tau(u_1,u_2)}]=E[e^{-s \tau(u_1,u_2)},T_1>h]+E[e^{-s \tau(u_1,u_2)},T_1 \leq
h].
\end{eqnarray}
For any $t\geq 0$, denote by $\mathcal{F}_t$ the information of the
two-dimensional risk process $\{(U_1(s),U_2(s)):s\geq 0\}$ up to
time $t$, and by $\theta_t$ the shift operator of the sample path,
i.e. $(\theta_t(\omega))_s=\omega_{s+t}$ for any sample path
$\omega=(\omega_s,s\geq 0)$. By the properties of conditional
expectation and the strong Markov property, we have
\begin{eqnarray}\label{thm3.1-b}
&&E[e^{-s \tau(u_1,u_2)},T_1>h]=E[e^{-s \tau(u_1,u_2)}
\textbf{1}_{\{T_1 > h\}}]\nonumber\\
&&=E[E[e^{-s \tau(u_1,u_2)} \textbf{1}_{\{T_1
> h\}}|\mathcal{F}_h]]\nonumber\\
&&=E[\textbf{1}_{\{T_1 > h\}} E[e^{-s [h+\tau \circ
\theta_h]}|\mathcal{F}_h]]\nonumber\\
&&=E\left[\textbf{1}_{\{T_1 > h\}} e^{-s h} E_{(U_1(h),U_2(h))}[e^{-s \tau}]\right]\nonumber\\
&&=\int_{h}^{\infty} e^{-s h} \Psi_{\rm min}\left(e^{r h} u_1+\frac{c_1}{r} (e^{r h}-1),e^{r h} u_2+\frac{c_2}{r} (e^{r h}-1),s\right)\lambda e^{-\lambda u} d u\nonumber\\
&&=e^{-(\lambda+s)h}\Psi_{\rm min}\left(e^{r h} u_1+\frac{c_1}{r} (e^{r h}-1),e^{r h} u_2+\frac{c_2}{r} (e^{r h}-1),s\right).
\end{eqnarray}
For the second item on the right side of (\ref{thm3.1-a}), we have
\begin{eqnarray}\label{thm3.1-c}
&&E[e^{-s \tau(u_1,u_2)},T_1 \leq h]\nonumber\\
&&=E\left[e^{-s
\tau(u_1,u_2)},T_1 \leq h, \sigma_1 \leq \frac{e^{r T_1} u_1+\frac{c_1}{r} (e^{r T_1}-1)}{\delta_1}
\wedge \frac{e^{r T_1} u_2+\frac{c_2}{r} (e^{r T_1}-1)}{\delta_2}\right]\nonumber\\
&&\ +E\left[e^{-s \tau},T_1 \leq h, \sigma_1 > \frac{e^{r T_1} u_1+\frac{c_1}{r} (e^{r T_1}-1)}{\delta_1}
\wedge \frac{e^{r T_1} u_2+\frac{c_2}{r} (e^{r T_1}-1)}{\delta_2}\right].
\end{eqnarray}
By the strong Markov property, we have
\begin{eqnarray}\label{thm3.1-d}
&&E\left[e^{-s
\tau(u_1,u_2)},T_1 \leq h, \sigma_1 \leq \frac{e^{r T_1} u_1+\frac{c_1}{r} (e^{r T_1}-1)}{\delta_1}
\wedge \frac{e^{r T_1} u_2+\frac{c_2}{r} (e^{r T_1}-1)}{\delta_2}\right]\nonumber\\
&&=E\left[e^{-s T_1}
E_{(U_1(T_1),U_2(T_1))}[e^{-s
\tau}],T_1 \leq h,\phantom{\int_0^1}\right.\nonumber\\
&&\quad\quad\quad\quad\left.\sigma_1 \leq \frac{e^{r T_1} u_1+\frac{c_1}{r} (e^{r T_1}-1)}{\delta_1}
\wedge \frac{e^{r T_1} u_2+\frac{c_2}{r} (e^{r T_1}-1)}{\delta_2}\right]\nonumber\\
&&=\int_{0}^{h} \lambda e^{-\lambda t} d t \int_{0}^{\frac{e^{r t} u_1+\frac{c_1}{r} (e^{r t}-1)}{\delta_1}
\wedge \frac{e^{r t} u_2+\frac{c_2}{r} (e^{r t}-1)}{\delta_2}} e^{-s t}\nonumber\\
&&\quad\times\Psi_{\rm min}\left(e^{r t} u_1+\frac{c_1}{r} (e^{r t}-1)-\delta_1 z,e^{r t} u_2+\frac{c_2}{r} (e^{r t}-1)-\delta_2 z,s\right) f(z)dz.
\end{eqnarray}
On the other hand, if $\sigma_1 > \frac{e^{r T_1} u_1+\frac{c_1}{r} (e^{r T_1}-1)}{\delta_1}
\wedge \frac{e^{r T_1} u_2+\frac{c_2}{r} (e^{r T_1}-1)}{\delta_2}$, then $\tau(u_1,u_2)=T_1$, and
thus
\begin{eqnarray}\label{thm3.1-e}
&&E\left[e^{-s \tau(u_1,u_2)},T_1 \leq h, \sigma_1 > \frac{e^{r T_1} u_1+\frac{c_1}{r} (e^{r T_1}-1)}{\delta_1}
\wedge \frac{e^{r T_1} u_2+\frac{c_2}{r} (e^{r T_1}-1)}{\delta_2}\right]\nonumber\\
&&=E\left[e^{-s T_1},T_1 \leq h, \sigma_1 > \frac{e^{r T_1} u_1+\frac{c_1}{r} (e^{r T_1}-1)}{\delta_1}
\wedge \frac{e^{r T_1} u_2+\frac{c_2}{r} (e^{r T_1}-1)}{\delta_2}\right]\nonumber\\
&&=\int_{0}^{h} \lambda e^{-\lambda t} d t\int_{\frac{e^{r t} u_1+\frac{c_1}{r} (e^{r t}-1)}{\delta_1}
\wedge \frac{e^{r t} u_2+\frac{c_2}{r} (e^{r t}-1)}{\delta_2}}^{\infty} e^{-s t} f(z) dz.
\end{eqnarray}
By (\ref{thm3.1-a})-(\ref{thm3.1-e}), we obtain
\begin{eqnarray}\label{thm3.1-f}
&&\Psi_{\rm min}(u_1,u_2,s)\nonumber\\
&&=e^{-(\lambda+s)h}\Psi_{\rm min}\left(e^{r h} u_1+\frac{c_1}{r} (e^{r h}-1),e^{r h} u_2+\frac{c_2}{r} (e^{r h}-1),s\right)\nonumber\\
&&\quad+\int_{0}^{h} \lambda e^{-\lambda t} d t\int_{0}^{\frac{e^{r t} u_1+\frac{c_1}{r} (e^{r t}-1)}{\delta_1}
\wedge \frac{e^{r t} u_2+\frac{c_2}{r} (e^{r t}-1)}{\delta_2}} e^{-s t}\nonumber\\
&&\quad\quad\quad\times\Psi_{\rm min}\left(e^{r t} u_1+\frac{c_1}{r} (e^{r t}-1)-\delta_1 z,e^{r t} u_2+\frac{c_2}{r} (e^{r t}-1)-\delta_2 z,s\right) f(z) dz\nonumber\\
&&\quad+\int_{0}^{h} \lambda e^{-\lambda t} d t\int_{\frac{e^{r t} u_1+\frac{c_1}{r} (e^{r t}-1)}{\delta_1}
\wedge \frac{e^{r t} u_2+\frac{c_2}{r} (e^{r t}-1)}{\delta_2}}^{\infty} e^{-s t} f(z)dz.
\end{eqnarray}
By the definition of $\Psi_{\rm min}(\cdot,\cdot,\cdot)$, we know that
if $z > \frac{e^{r t} u_1+\frac{c_1}{r} (e^{r t}-1)}{\delta_1}
\wedge \frac{e^{r t} u_2+\frac{c_2}{r} (e^{r t}-1)}{\delta_2}$, then
$\Psi_{\rm min}(e^{r t} u_1+\frac{c_1}{r} (e^{r t}-1)-\delta_1 z,e^{r t} u_2+\frac{c_2}{r} (e^{r t}-1)-\delta_2 z,s)=1. $ By
virtue of this fact and letting $y:=e^{r h}-1$, $q_1:=u_1+\frac{c_1}{r}$ and $q_2:=u_2+\frac{c_2}{r}$, we can rewrite
(\ref{thm3.1-f}) by
\begin{eqnarray}\label{thm3.1-g}
&&\Psi_{\rm min}(u_1,u_2,s)\nonumber\\
&&=e^{-(\lambda+s)h}\Psi_{\rm min}\left(u_1+q_1 y,u_2+q_2 y,s\right)\nonumber\\
&&\quad+\int_{0}^{h} \lambda e^{-\lambda t} d t\int_{0}^{\infty} e^{-st}\nonumber\\
&&\quad\quad\times\Psi_{\rm min}\left(e^{r t} u_1+\frac{c_1}{r} (e^{r t}-1)-\delta_1 z,e^{r t} u_2+\frac{c_2}{r} (e^{r t}-1)-\delta_2 z,s\right)f(z)dz.\quad\quad\quad
\end{eqnarray}
It's easy to check that  $y \uparrow 0$ if and only if $h \downarrow 0$. Hence by
(\ref{thm3.1-g}), we have
\begin{eqnarray}\label{thm3.1-h}
\lim_{y\uparrow 0}\Psi_{\rm min}(u_1+q_1 y,u_2+q_2 y,s)=
\Psi_{\rm min}(u_1,u_2,s).
\end{eqnarray}
By (\ref{thm3.1-g}), for any $h>0$ and $y=e^{r h}-1$, we have
\begin{eqnarray*}
0&=&\frac{\Psi_{\rm min}(u_1+q_1 y,u_2+q_2
y,s)-\Psi_{\rm min}(u_1,u_2,s)}{y}\nonumber\\
&&+\frac{e^{-(\lambda+s)h}-1}{y}\Psi_{\rm min}(u_1+q_1
y,u_2+q_2
y,s)\nonumber\\
&&+\frac{1}{y}\int_{0}^{h} \lambda e^{-\lambda t} d
t\int_{0}^{\infty} e^{-s t}\nonumber\\
&&\quad\quad\quad\times\Psi_{\rm min}\left(e^{r t} u_1+\frac{c_1}{r} (e^{r t}-1)-\delta_1 z,e^{r t} u_2+\frac{c_2}{r} (e^{r t}-1)-\delta_2 z,s\right) f(z) d z\nonumber
\end{eqnarray*}
\begin{eqnarray}\label{thm3.1-i}
&=&\frac{\Psi_{\rm min}(u_1+q_1 y,u_2+q_2
y,s)-\Psi_{\rm min}(u_1,u_2,s)}{y}\nonumber\\
&&+\frac{e^{-(\lambda+s) h}-1}{e^{r
h}-1}\Psi_{\rm min}(u_1+q_1
y,u_2+q_2 y,s)\nonumber\\
&&+\frac{1}{e^{r h}-1}\int_{0}^{h} \lambda e^{-\lambda t} d
t\int_{0}^{\infty} e^{-s t}\nonumber\\
&&\quad\times\Psi_{\rm min}\left(e^{r t} u_1+\frac{c_1}{r} (e^{r t}-1)-\delta_1 z,e^{r t} u_2+\frac{c_2}{r} (e^{r t}-1)-\delta_2 z,s\right) f(z) d z.\quad\quad\quad\label{extra}
\end{eqnarray}
By (\ref{thm3.1-h}), letting $y \uparrow0,h \downarrow 0$  in the above
formula and noticing that (\ref{bound}) assures the interchange of limitation and integration, we obtain
\begin{eqnarray}\label{thm3.1-i}
&&q_1 \frac{\partial \Psi_{\rm min}}{\partial u_1}+q_2 \frac{\partial
\Psi_{\rm min}}{\partial u_2}-\frac{\lambda+s}{r}
\Psi_{\rm min}\nonumber\\
&&\hskip 1.45cm+\frac{\lambda}{r} \int_{0}^{\infty}
\Psi_{\rm min}(u_1-\delta_1 z,u_2-\delta_2 z,s)f(z)d z=0.
\end{eqnarray}
Replacing $q_1$ and $q_2$ in (\ref{thm3.1-i}) by $u_1+\frac{c_1}{r}$ and $u_2+\frac{c_2}{r}$ respectively, we obtain the integral-differential equation. When $u_1/\delta_1 = u_2/\delta_2$, the joint ruin model degenerates into a univariate model, and then by the  analysis in \cite{APP08}, we get  the boundary condition.

{\bf Uniqueness:} By using similar arguments in Gerber \cite{Ge81}
 and noticing (\ref{thm3.1-g}), we define an
operator $\mathcal{T}$ by
\begin{eqnarray*}
\mathcal{T} g(u_1,u_2,s)&=&e^{-(\lambda+s)h}g\left(u_1+q_1 y,u_2+q_2 y,s\right)\nonumber\\
&&+\int_{0}^{h} \lambda e^{-\lambda t} d t\int_{0}^{\infty} e^{-s
t}\nonumber\\
&&\times g\left(e^{r t} u_1+\frac{c_1}{r} (e^{r t}-1)-\delta_1 z,e^{r t} u_2+\frac{c_2}{r} (e^{r t}-1)-\delta_2 z,s\right)f(z)dz,
\end{eqnarray*}
for any $h>0$. It can be easily seen that $\Psi_{\rm min}$ is a fixed point of operator $\mathcal{T}$, as $\mathcal{T} \Psi_{\rm min}=\Psi_{\rm min}$. Also, for two different functions $g_1$ and $g_2$ we have for any $h>0$ and $s>0$,
\begin{eqnarray}\label{thm3.1-j}
&&|\mathcal{T}g_1-\mathcal{T}g_2|\nonumber\\
&&\leq e^{-(\lambda+s)h}|g_1\left(u_1+q_1 y,u_2+q_2 y,s\right)-
g_2\left(u_1+q_1 y,u_2+q_2 y,s\right)|\nonumber\\
&&\ \ +\int_{0}^{h} \lambda e^{-\lambda t} d t\int_{0}^{\infty}
e^{-s
t}\left|g_1\left(e^{r t} u_1+\frac{c_1}{r} (e^{r t}-1)-\delta_1 z,e^{r t} u_2+
\frac{c_2}{r} (e^{r t}-1)-\delta_2 z,s\right)\right.\nonumber\\
&&\ \ \left.-g_2\left(e^{r t} u_1+\frac{c_1}{r}
(e^{r t}-1)-\delta_1 z,e^{r t} u_2+\frac{c_2}{r} (e^{r t}-1)-\delta_2 z,s\right)\right|f(z)dz\nonumber\\
&&\leq e^{-(\lambda+s)h}||g_1-g_2||_{\infty}+\left(\int_{0}^{h} \lambda e^{-(\lambda+s) t}dt\right) ||g_1-g_2||_{\infty}\nonumber\\
&&=\frac{\lambda+se^{-(\lambda+s)h}}{\lambda+s}||g_1-g_2||_{\infty},
%&&<||g_1-g_2||_{\infty}\nonumber,
\end{eqnarray}
where $||\cdot||_{\infty}$ is the  supremum norm over $(u_1,u_2)\in
R^2$. Therefore, $\mathcal{T}$ is a contraction and by Banach's fixed point theorem and (\ref{bound}), the solution of  (\ref{thm3.1-0})-(\ref{thm3.1-k}) is unique.
\end{proof}

\begin{rem}
One way to obtain the Laplace transform $\Psi_{\rm min}$ of the joint ruin
probability $T_{\rm min}(u_1,u_2)$ is to solve the above integral-differential equation (\ref{thm3.1-0})-(\ref{thm3.1-k})  numerically. A natural question arise:

{\it Can we give an analytical representation for the solution to the equation (\ref{thm3.1-0})-(\ref{thm3.1-k}) in some special cases such as exponential claim sizes ?}

Unfortunately, even in the case of exponential claim sizes, we have not found the way to solve the equation (\ref{thm3.1-0})-(\ref{thm3.1-k}).  \end{rem}

\subsection{The result about $T_{\rm max}(u_1,u_2)$}

Define the  Laplace transform of $T_{\rm max}(u_1,u_2)$  by
\begin{eqnarray*}
\Psi_{\rm max}(u_1,u_2,s):= E\left[e^{-s T_{\rm max}(u_1,u_2)}\right], \quad\mbox{for}\ s>0.
\end{eqnarray*}
Then we have the following result.

\begin{thm}\label{thm2.2}
For $\frac{u_1}{\delta_1} \leq \frac{u_2}{\delta_2}$ and $s>0$, the function $\Psi_{\rm max}(\cdot,\cdot,s)$ satisfies the same integral-differential equation (\ref{thm3.1-0}) with the boundary condition
\begin{eqnarray}\label{thm2.2-a}
\Psi_{\rm max}(u_1,\frac{\delta_2}{\delta_1}u_1,s)= E\left[e^{-s \tau_1(u_1)}\right],
\end{eqnarray}
where $f(z)$ is the probability density function of $\sigma_k$ and $\tau_1$ is the ruin time of risk process $U_1(t)$. Furthermore, $\Psi_{\rm max}$ is the unique solution of (\ref{thm3.1-0})-(\ref{thm2.2-a}).
\end{thm}
\begin{proof}
The proof is almost  the same with that of Theorem \ref{thm3.1}, and we need only to notice the following three things:
\begin{itemize}
\item[(1)] In this case, (\ref{thm3.1-c}) becomes
\begin{eqnarray*}
&&E[e^{-s \tau(u_1,u_2)},T_1 \leq h]\nonumber\\
&&=E\left[e^{-s
\tau(u_1,u_2)},T_1 \leq h, \sigma_1 \leq \frac{e^{r T_1} u_1+\frac{c_1}{r} (e^{r T_1}-1)}{\delta_1}
\vee \frac{e^{r T_1} u_2+\frac{c_2}{r} (e^{r T_1}-1)}{\delta_2}\right]\nonumber\\
&&\ +E\left[e^{-s \tau},T_1 \leq h, \sigma_1 > \frac{e^{r T_1} u_1+\frac{c_1}{r} (e^{r T_1}-1)}{\delta_1}
\vee \frac{e^{r T_1} u_2+\frac{c_2}{r} (e^{r T_1}-1)}{\delta_2}\right],
\end{eqnarray*}
where $\tau(u_1,u_2)$ stands for $T_{\rm max}(u_1,u_2)$.

\item[(2)] $\tau(u_1,u_2)=T_1$, if $\sigma_1 > \frac{e^{r T_1} u_1+\frac{c_1}{r} (e^{r T_1}-1)}{\delta_1}
\vee \frac{e^{r T_1} u_2+\frac{c_2}{r} (e^{r T_1}-1)}{\delta_2}$.

\item[(3)] If $z > \frac{e^{r t} u_1+\frac{c_1}{r} (e^{r t}-1)}{\delta_1}
\vee \frac{e^{r t} u_2+\frac{c_2}{r} (e^{r t}-1)}{\delta_2}$, then
$\Psi_{\rm max}(e^{r t} u_1+\frac{c_1}{r} (e^{r t}-1)-\delta_1 z,e^{r t} u_2+\frac{c_2}{r} (e^{r t}-1)-\delta_2 z,s)=1$.
\end{itemize}
We omit the details.
\end{proof}

\section{Asymptotics for finite time ruin probabilities}
In this section, we consider the finite time ruin probability
associated with  $T_{\rm max}(u_1,u_2)$ and $T_{\rm min}(u_1,u_2)$. The original idea comes from \cite[Section 4]{LLT07}.

Define $X_i(t):=e^{-r t} U_i(t)/\delta_i,i=1,2$. Then
$(X_1(t),X_2(t))$ has the same ruin times and probabilities with
$(U_1(t),U_2(t))$. Denote $x_i:=\frac{u_i}{\delta_i},
p_i:=\frac{c_i}{r \delta_i}, i=1,2.$ Then by (\ref{model-2}) and our assumptions, we
have
\begin{eqnarray}
X_i(t)=x_i+p_i (1-e^{-r t})-\sum\limits_{k=1}^{N(t)} e^{-r \theta_k}
\sigma_k,\  i=1,2,
\end{eqnarray}
where $p_1>p_2, x_1\leq x_2$.

For $T>0$, define
$
\Psi_{\rm max}(x_1,x_2,T):=P\{T_{\rm max}(\delta_1 x_1,\delta_2 x_2)\leq T\}.
$ Then  we have
\begin{eqnarray}\label{4-1}
\Psi_{\rm max}(x_1,x_2,T)=P\{\exists t\leq  T\ s.t.\ X_1(t)<0\
\mbox{and}\ X_2(t)<0\}.
\end{eqnarray}

Alternatively, we can also define $
\Psi_{\rm min}(x_1,x_2,T):=P\{T_{\rm min}(\delta_1 x_1,\delta_2 x_2)\leq T\}
$ and get
\begin{eqnarray}\label{4-1b}
\Psi_{\rm min}(x_1,x_2,T)=P\{\exists t\leq  T\ s.t.\ X_1(t)<0\
\mbox{or}\ X_2(t)<0\}.
\end{eqnarray}

In the following, we will provide asymptotic results on both $\Psi_{\rm max}(x_1,x_2,T)$ and $\Psi_{\rm min}(x_1,x_2,T)$ under some condition.

\subsection{Asymptotic result about $T_{\rm max}(u_1,u_2)$}

Let $T>0$, $n\in \mathbf{N}$, and $\{V_k,k=1,2,\ldots,n\}$ be a
sequence of i.i.d. random variables with the uniform distribution on
$(0,T]$. Denote by $(V_1^*,\ldots,V_n^*)$ the ordered statistic of
$(V_1,...,V_n)$. It's well known that conditioning on $\{N(t)=n\}$,
the random vectors $(\theta_1,...,\theta_n)$ and
$(V_1^*,\ldots,V_n^*)$ have the same distribution. Assume that
$\{V_k,k=1,2,\ldots,n\}$ is independent of $\{\sigma_k,k\geq 1\}$.
Define $F_T(x)=P[e^{-r V_1} \sigma_1 \leq x]$. Then we have
\begin{eqnarray}\label{4-2}
&&P\left\{\left.\sum\limits_{k=1}^{n} e^{-r \theta_k} \sigma_k >x
\right|N(T)=n\right\}=P\left\{\sum\limits_{k=1}^{n} e^{-r V_k^*}
\sigma_k >
x\right\}\nonumber\\
&&=P\left\{\sum\limits_{k=1}^{n} e^{-r V_k} \sigma_k >
x\right\}=\overline{{F_T}^{*n}}(x),
\end{eqnarray}
where ${F_T}^{*n}(x)$ stands for the $n$-multiple convolution of
$F_T(x)$.

\begin{thm}\label{thm4.1}
If $\sigma_k$ has a regularly varying tail with
$P\{\sigma_k>x\}=L(x)/x^{\alpha}$, where $L$ is continuous, slowly
varying, $\lim\limits_{x\to\infty}L(x)=\infty$, and $\alpha>0$. Then for any
$T>0$, we have
\begin{eqnarray}\label{thm4.2-0}
 \lim\limits_{x_2\geq x_1\rightarrow
\infty}\frac{\Psi_{\rm max}(x_1,x_2,T)}{\lambda T \overline{{F_T}}(x_2)}
= 1.
\end{eqnarray}
\end{thm}

Before proving Theorem \ref{thm4.1},  we need one lemma.

%\begin{lem}\label{lem4.1}
%For any  $t>0$, it holds that
%\begin{eqnarray}\label{lem4.1-a}
%\liminf\limits_{x\rightarrow \infty}
%\frac{\overline{{F_t}^{*n}}(x)}{\overline{{F_t}}(x)}\geq n.
%\end{eqnarray}
%\end{lem}
%{\bf Proof.} By \cite[Proposition IX.1.1(b)]{As00}, following the
%proof of \cite[Proposition IX.1.7]{As00}, we can easily obtain (\ref{lem4.1-a}).\hfill\fbox

\begin{lem}\label{lem3.3}
Suppose that  $\sigma_k$ satisfies the condition in Theorem \ref{thm4.1}. Then $F_T$
has a regularly varying tail.
\end{lem}
\begin{proof} By the independence of $V_1$ and $\sigma_1$, we have
\begin{eqnarray*}
\overline{F_T}(x)&=&P\{e^{-rV_1}\sigma_1>x\}\\
&=&\int_0^TP\{e^{-ry}\sigma_1>x\}\frac{1}{T}dy\\
&=&\frac{1}{T}\int_0^T\frac{L(e^{ry}x)}{(e^{ry}x)^{\alpha}}dy:=\frac{S(x)}{x^{\alpha}},
\end{eqnarray*}
where
$
S(x)=\frac{1}{T}\int_0^T\frac{L(e^{ry}x)}{(e^{ry})^{\alpha}}dy,
$
which together with the assumption that $L$ is continuous and $\lim\limits_{x\to\infty}L(x)=\infty$
implies that
\begin{eqnarray}\label{thm3.3-a}
\lim_{x\to\infty}S(x)=\infty.
\end{eqnarray}
By the change of variable, we get that
\begin{eqnarray}\label{thm3.3-b}
S(x)=\frac{x^{\alpha}}{rT}\int_x^{e^{rT}x}\frac{L(u)}{u^{\alpha+1}}du.
\end{eqnarray}
For any $t>0$, by (\ref{thm3.3-a}), (\ref{thm3.3-b}) and the fact that $L$ is a slowing varying function,  we obtain
\begin{eqnarray*}
&&\lim_{x\to\infty}\frac{S(tx)}{S(x)}=
\lim_{x\to\infty}\frac{t^\alpha\int_{tx}^{e^{rT}tx}\frac{L(u)}{u^{\alpha+1}}du}
{\int_x^{e^{rT}x}\frac{L(u)}{u^{\alpha+1}}du}=\lim_{x\to\infty}\frac{t^{\alpha}\left(\frac{L(e^{rT}tx)}{(e^{rT}tx)^{\alpha+1}}e^{rT}t-
\frac{L(tx)}{(tx)^{\alpha+1}}t\right)
}{\frac{L(e^{rT}x)}{(e^{rT}x)^{\alpha+1}}e^{rT}-
\frac{L(x)}{x^{\alpha+1}}}\\
&&=\lim_{x\to\infty}\frac{\frac{L(e^{rT}tx)}{(e^{rT})^{\alpha}}-
L(tx) }{\frac{L(e^{rT}x)}{(e^{rT})^{\alpha}}-L(x)}=\lim_{x\to\infty}\frac{\frac{L(e^{rT}tx)}{L(e^{rT}x)}-(e^{rT})^{\alpha}\frac{L(tx)}{L(e^{rT}x)}}
{1-(e^{rT})^{\alpha}\frac{L(x)}{L(e^{rT}x)}}\\
&&=\frac{1-(e^{rT})^{\alpha}}{1-(e^{rT})^{\alpha}}=1.
\end{eqnarray*}
Hence $F_T$ has a regularly varying tail.
\end{proof}

\noindent{\bf Proof of Theorem \ref{thm4.1}.}
By Lemma \ref{lem3.3}
and \cite[Proposition IX.1.4]{As00}, we know that $F_T$  is a
subexponential distribution.  By (\ref{4-1}) and (\ref{4-2}), we
have
\begin{eqnarray}\label{thm4.2-a}
&&\Psi_{\rm max}(x_1,x_2,T)\nonumber\\
&&=P\left\{\sum\limits_{k=1}^{N(t)} e^{-r \theta_k} \sigma_k>
x_i+p_i (1-e^{-r t}),i=1,2,\ \mbox{for some}\
t \leq T\right\}\nonumber\\
&&\geq P\left\{\sum\limits_{k=1}^{N(T)} e^{-r \theta_k} \sigma_k >
x_i+p_i
(1-e^{-rT}),i=1,2\right\}\nonumber\\
&&=\sum\limits_{n=0}^{\infty} P\{N(T)=n\}\nonumber\\
&&\quad\times P\left\{\left.\sum\limits_{k=1}^{n} e^{-r \theta_k} \sigma_k >
x_i+p_i (1-e^{-r T}),i=1,2\right|N(T)=n\right\}.
\end{eqnarray}

If $x_1+p_1 (1-e^{-r T}) \geq x_2+p_2 (1-e^{-r T})$, then by
(\ref{4-2}) and the assumption that $x_2\geq x_1$, we obtain
\begin{eqnarray}\label{thm4.2-b}
&&P\left\{\left.\sum\limits_{k=1}^{n} e^{-r \theta_k} \sigma_k >
x_i+p_i (1-e^{-r T}),i=1,2\right|N(T)=n\right\}\nonumber\\
&&=P\left\{\left.\sum\limits_{k=1}^{n} e^{-r \theta_k} \sigma_k >
x_1+p_1 (1-e^{-r T})\right|N(T)=n\right\}\nonumber\\
&&=\overline{{F_T}^{*n}}(x_1+p_1 (1-e^{-r T})),
\end{eqnarray}
and $x_2+p_1  (1-e^{-r T})\geq x_1+p_1  (1-e^{-r T})\geq x_2+p_2(1-e^{-r T})>x_2$, which implies that
\begin{eqnarray*}
\overline{{F_T}^{*n}}(x_ 2+p_1  (1-e^{-r T})) \leq \overline{{F_T}^{*n}}(x_1+p_1 (1-e^{-r T})) \leq
\overline{{F_T}^{*n}}(x_2),
\end{eqnarray*}
and thus
\begin{eqnarray}\label{thm4.2-c}
\frac{\overline{{F_T}^{*n}}(x_2+p_1 (1-e^{-r
T}))}{\overline{{F_T}^{*n}}(x_2)}\leq \frac{\overline{{F_T}^{*n}}(x_1+p_1 (1-e^{-r
T}))}{\overline{{F_T}^{*n}}(x_2)} \leq 1.
\end{eqnarray}
Since $F_T$  is a
subexponential distribution, by \cite[Proposition IX.1.5]{As00} and (\ref{thm4.2-c}), it holds that
\begin{eqnarray}\label{thm4.2-d}
\liminf\limits_{x_1 \rightarrow \infty}
\frac{\overline{{F_T}^{*n}}(x_1+p_1 (1-e^{-r
T}))}{\overline{{F_T}^{*n}}(x_2)}=1.
%=\liminf\limits_{x_1 %\rightarrow\infty}\frac{\overline{{F_T}^{*n}}(x_1)}{\overline{{F_T}^{*n}}(x_2)},
\end{eqnarray}
By  Fatou's Lemma, (\ref{thm4.2-d}) and   \cite[Proposition IX.1.7]{As00}, we have
\begin{eqnarray}\label{thm4.2-e}
&&\liminf\limits_{x_1\rightarrow
\infty}\frac{\sum\limits_{n=0}^{\infty} P\{N(T)=n\}
\overline{{F_T}^{*n}}(x_1+p_1 (1-e^{-r T}))}{\lambda T \overline{{F_T}}(x_2)}\nonumber\\
&&=\liminf\limits_{x_1 \rightarrow
\infty}\sum\limits_{n=0}^{\infty} P\{N(T)=n\} \frac{\overline{{F_T}^{*n}}(x_1+p_1 (1-e^{-r T}))}{\overline{{F_T}^{*n}}(x_2)} \frac{\overline{{F_T}^{*n}}(x_2)}{\lambda T \overline{{F_T}}(x_2)}\nonumber\\
&& \geq \frac{1}{\lambda T}\sum\limits_{n=0}^{\infty} P\{N(T)=n\} \liminf\limits_{x_1 \rightarrow \infty} \frac{\overline{{F_T}^{*n}}(x_1+p_1 (1-e^{-r T}))}{\overline{{F_T}^{*n}}(x_2)}\liminf\limits_{x_1 \rightarrow \infty}\frac{\overline{{F_T}^{*n}}(x_2)}{\overline{{F_T}}(x_2)}\nonumber\\
&&=\frac{1}{\lambda T}\sum\limits_{n=0}^{\infty} P\{N(T)=n\} \liminf\limits_{x_1 \rightarrow \infty} \frac{\overline{{F_T}^{*n}}(x_2)}{\overline{{F_T}}(x_2)}\nonumber\\
&&= \frac{1}{\lambda T}\sum\limits_{n=0}^{\infty} P\{N(T)=n\} n\nonumber\\
&&=\frac{1}{\lambda T}E[N(t)]=1.
\end{eqnarray}
By (\ref{thm4.2-a}), (\ref{thm4.2-b}) and (\ref{thm4.2-e}) and under the condition that $x_1+p_1 (1-e^{-r T}) \geq x_2+p_2 (1-e^{-r T})$, we have that
\begin{eqnarray*}
\liminf\limits_{x_1\rightarrow
\infty}\frac{\Psi_{max}(x_1,x_2,T)}{\lambda T \overline{{F_T}}(x_2)}
\geq 1.
\end{eqnarray*}

If $x_1+p_1 (1-e^{-r T}) < x_2+p_2 (1-e^{-r T})$, then
\begin{eqnarray}\label{thm4.2-f}
&&P\left[\left.\sum\limits_{k=1}^{n} e^{-r \theta_k} \sigma_k > x_i+p_i
(1-e^{-r T}),i=1,2\right|N(T)=n\right]\nonumber\\
&&=P\left[\left.\sum\limits_{k=1}^{n} e^{-r \theta_k} \sigma_k > x_2+p_2
(1-e^{-r T})\right|N(T)=n\right]\nonumber\\
&&=\overline{{F_T}^{*n}}(x_2+p_2 (1-e^{-r T})).
\end{eqnarray}
Since $F_T$  is a
subexponential distribution and $x_2\geq x_1$, by \cite[Proposition IX.1.5]{As00} we have
\begin{eqnarray}\label{thm4.2-g}
\lim_{x_1\to\infty}\frac{\overline{{F_T}^{*n}}(x_2+p_2 (1-e^{-r
T}))}{\overline{{F_T}^{*n}}(x_2)}=1.
\end{eqnarray}
Now  By (\ref{thm4.2-a}), (\ref{thm4.2-f}) and (\ref{thm4.2-g}), similar to the arguments in (\ref{thm4.2-e}), we obtain that  under the condition that $x_1+p_1 (1-e^{-r T})< x_2+p_2 (1-e^{-r T})$
\begin{eqnarray*}
\liminf\limits_{x_1\rightarrow
\infty}\frac{\Psi_{max}(x_1,x_2,T)}{\lambda T \overline{{F_T}}(x_2)}
\geq 1.
\end{eqnarray*}
Hence we always have
\begin{eqnarray}\label{thm4.2-h}
\liminf\limits_{x_1\rightarrow
\infty}\frac{\Psi_{max}(x_1,x_2,T)}{\lambda T \overline{{F_T}}(x_2)}
\geq 1.
\end{eqnarray}

On the other hand, by the assumption that $x_2\geq x_1$, and (\ref{4-2}), we have
\begin{eqnarray*}
&&\Psi_{\rm max}(x_1,x_2,T)\\
&&=P\left\{\sum\limits_{k=1}^{N(t)} e^{-r \theta_k} \sigma_k>
x_i+p_i (1-e^{-r t}),i=1,2,\ \mbox{for some}\
t \leq T\right\}\\
&&\leq P\left\{\sum\limits_{k=1}^{N(T)} e^{-r \theta_k} \sigma_k >
x_2\right\}\\
&&=\sum\limits_{n=0}^{\infty} P\{N(T)=n\}
P\left\{\left.\sum\limits_{k=1}^{n} e^{-r \theta_k} \sigma_k >
x_2\right|N(T)=n\right\}\\
&&=\sum\limits_{n=0}^{\infty} P\{N(T)=n\}\overline{{F_T}^{*n}}(x_2).
\end{eqnarray*}
By Fatou's Lemma, the above formula and \cite[Proposition IX.1.7]{As00}, we have
\begin{eqnarray}\label{thm4.2-i}
&&\limsup\limits_{x_1\rightarrow
\infty}\frac{\Psi_{max}(x_1,x_2,T)}{\lambda T \overline{{F_T}}(x_2)}\nonumber\\
&&\leq\limsup\limits_{x_1\rightarrow
\infty}\frac{\sum\limits_{n=0}^{\infty} P\{N(T)=n\}
\overline{{F_T}^{*n}}(x_2)}{\lambda T \overline{{F_T}}(x_2)}\nonumber\\
&& \leq \frac{1}{\lambda T}\sum\limits_{n=0}^{\infty} P\{N(T)=n\} \limsup\limits_{x_1 \rightarrow \infty} \frac{\overline{{F_T}^{*n}}(x_2)}{\overline{{F_T}}(x_2)}\nonumber\\
&&\leq \frac{1}{\lambda T}\sum\limits_{n=0}^{\infty} P\{N(T)=n\} n\nonumber\\
&&=\frac{1}{\lambda T}E[N(t)]=1.
\end{eqnarray}
It follows from (\ref{thm4.2-h}) and (\ref{thm4.2-i}) that (\ref{thm4.2-0}) holds.

\subsection{Asymptotic result about $T_{\rm min}(u_1,u_2)$}

By Theorem \ref{thm4.1} we can easily obtain the asymptotic result for $\Psi_{\rm min}(x_1,x_2,T)$, which is formulated as follow:

\begin{thm}\label{cor4.1}
If $\sigma_k$ has a regularly varying tail with
$P\{\sigma_k>x\}=L(x)/x^{\alpha}$, where $L$ is continuous, slowly
varying, $\lim\limits_{x\to\infty}L(x)=\infty$, and $\alpha>0$. Then for any
$T>0$, we have
\begin{eqnarray}\label{thm4.2-00}
\lim\limits_{x_2\geq x_1\rightarrow
\infty}\frac{\Psi_{\rm min}(x_1,x_2,T)}{\lambda T \overline{{F_T}}(x_1)}
= 1.
\end{eqnarray}
\end{thm}

\begin{proof}
First, for $i=1,2$, define
\begin{eqnarray*}
\psi_i (x_i, T)= P\{\exists \ t\leq  T\ s.t.\ X_i(t)<0\},
\end{eqnarray*}
i.e. $\psi_i (x_i, T)(i=1,2)$ represents the ruin probability of $X_i(t)(i=1,2)$ within finite time $T$.

Notice the fact that
\begin{eqnarray*}\label{4-2}
&&P\{\exists t\leq  T\ s.t.\ X_1(t)<0\ \mbox{and}\ X_2(t)<0\} \nonumber\\
&&=P\{\exists t\leq  T\ s.t.\ X_1(t)<0\}+P\{\exists t\leq  T\ s.t.\ X_2(t)<0\}-P\{\exists t\leq  T\ s.t.\ X_1(t)<0\
\mbox{or}\ X_2(t)<0\}.
\end{eqnarray*}
Then by (\ref{4-1}) and (\ref{4-1b}), we have
\begin{eqnarray}\label{4-22}
\Psi_{\rm max}(x_1,x_2,T)=\psi_1 (x_1, T)+\psi_2 (x_2, T)-\Psi_{\rm min}(x_1,x_2,T).
\end{eqnarray}

%For $\psi_i (x_i, T)(i=1,2)$, we have
%\begin{eqnarray*}
%&&\sum\limits_{n=0}^{\infty} P\{N(T)=n\}\overline{{F_T}^{*n}}(x_i)\\
%&&=\sum\limits_{n=0}^{\infty} P\{N(T)=n\}
%P\left\{\left.\sum\limits_{k=1}^{n} e^{-r \theta_k} \sigma_k >
%x_i\right|N(T)=n\right\}=P\left\{\sum\limits_{k=1}^{N(T)} e^{-r \theta_k} \sigma_k>
%x_i\right\} \\
%&&\geq P\left\{\exists \ t \leq T\  s.t.\ \sum\limits_{k=1}^{N(t)} e^{-r \theta_k} \sigma_k>
%x_i+p_i (1-e^{-r t})\right\}\\
%&&=P\{\exists t\leq  T\ s.t.\ X_i(t)<0\}=\psi_i (x_i, T)\\
%&&\geq P\left\{\sum\limits_{k=1}^{N(T)} e^{-r \theta_k} \sigma_k >
%x_i+p_i
%(1-e^{-rT}) \right\}\\
%&&=\sum\limits_{n=0}^{\infty} P\{N(T)=n\}P\left\{\left.\sum\limits_{k=1}^{n} e^{-r \theta_k} \sigma_k >
%x_i+p_i (1-e^{-r T})\right|N(T)=n\right\}\\
%&&=\sum\limits_{n=0}^{\infty} P\{N(T)=n\}\overline{{F_T}^{*n}}(x_i+p_i (1-e^{-r T})).
%\end{eqnarray*}
By Lemma \ref{lem3.3},  $F_T$  is a subexponential distribution. Then by \cite[Proposition IX.1.5]{As00}, for $i=1,2$, we have
\begin{eqnarray}\label{4-1c}
 \lim\limits_{x_2\geq x_1\rightarrow
\infty}\frac{\psi_i(x_i,T)}{\lambda T \overline{{F_T}}(x_i)}
= 1.
\end{eqnarray}

 By  (\ref{4-22}), (\ref{4-1c}), (\ref{thm4.2-0}), and the fact that $x_2 \geq x_1$, we obtain that
\begin{eqnarray*}
&&\left|\frac{\Psi_{\rm min}(x_1,x_2,T)}{\lambda T \overline{{F_T}}(x_1)} -1\right|\\
&&=\left|\frac{\psi_1(x_1,T)-\lambda T \overline{{F_T}}(x_1)+\psi_2(x_2,T)-\Psi_{\rm max}(x_1,x_2,T)}{\lambda T \overline{{F_T}}(x_1)}\right|\\
&&\leq \left|\frac{\psi_1(x_1,T)-\lambda T \overline{{F_T}}(x_1)}{\lambda T \overline{{F_T}}(x_1)}\right|+\left|\frac{\psi_2(x_2,T)-\Psi_{\rm max}(x_1,x_2,T)}{\lambda T \overline{{F_T}}(x_2)}\right|\cdot\left|\frac{\overline{{F_T}}(x_2)}
{\overline{{F_T}}(x_1)}\right|\\
&&\leq\left|\frac{\psi_1(x_1,T)-\lambda T \overline{{F_T}}(x_1)}{\lambda T \overline{{F_T}}(x_1)}\right|+\left|\frac{\psi_2(x_2,T)-\Psi_{\rm max}(x_1,x_2,T)}{\lambda T \overline{{F_T}}(x_2)}\right|\rightarrow 0,\ \mbox{as}\ x_2\geq x_1\rightarrow \infty.
\end{eqnarray*}
\end{proof}

\acks

The authors acknowledge the helpful suggestions and comments of one anonymous referee,
which helped improve the first version of this manuscript. We are grateful to the support of NNSFC (Grant No. 10801072).

% Place the text of your acknowledgements after the \acks command.
% \acks generates the heading "Acknowledgements".
% If you wish to make only one acknowledgement, use \ack.
% \ack generates the heading "Acknowledgement".

% Reference list
%
% References should be in the following form (or the BibTeX file
% apt.bst should be used):
%
% For a journal:
% Surname, Initial (year). Title of paper. {\em Journal title}
% {\bf Vol,} page--range.
%
% For a book:
% Surname, Initial (year). {\em Book title}. Publisher, Address.
%
% Note the following example of a reference list.

\end{document}